\documentclass[12pt]{amsart}
\usepackage{amsmath,amssymb}

\newtheorem{theorem}{Theorem}[section]

\newtheorem{lemma}[theorem]{Lemma}
\newtheorem{proposition}[theorem]{Proposition}

\theoremstyle{definition}

\newtheorem{example}[theorem]{Example}

\theoremstyle{remark}
\newtheorem{remark}[theorem]{Remark}

\def \bC{\mathbb C}

\def\bB{\mathbb B}
\def\bH{\mathbb H}
\def\bCP{\mathbb C\mathbb P}
\def\im{\text{{\rm Im }}}
\def\re{\text{{\rm Re }}}

\title[Proper holomorphic mappings into $\ell$-concave quadric domains]{Proper holomorphic mappings into $\ell$-concave quadric domains in projective space}
\author{Peter Ebenfelt}
\address{Department of Mathematics, University of California at San Diego, La Jolla, CA 92093-0112}
\email{pebenfel@math.ucsd.edu}
\thanks{The author was partly supporting by the NSF grant DMS-1301282.}
\begin{document}
\begin{abstract} In this paper, we prove a type of partial rigidity result for proper holomorphic mappings of certain $\ell$-concave domains in projective space into model quadratic $\ell$-concave domains. The main technical result is a degree estimate for proper holomorphic mappings into the model domains, provided that the mappings extend to projective space as rational mappings, and the source domain contains algebraic varieties and has a boundary with low CR complexity.
\end{abstract}
\maketitle

\section{Introduction}

The study of holomorphic mappings (or, more generally, CR mappings) sending a Levi nondegenerate hypersurface $M$ in an $(n+1)$-dimensional complex manifold into
a hyperquadric in $\bCP^{N+1}$, $N\geq n$, has attracted considerable attention since Alexander in 1974 \cite{Alexander74} rediscovered Poincar\`e's classical rigidity
result \cite{Poincare07} for local mappings between spheres. In recent years, much progress has been made on understanding when rigidity
(uniqueness modulo autmorphisms of the hyperquadric) holds in the case where the Levi signatures of $M$ and the target hyperquadric are the equal (see, e.g., \cite{Webster79}, \cite{Faran86}, \cite{Huang99}, \cite{EHZ04}, \cite{EHZ05}, \cite{BH05}, \cite{BEH08}). Even more recently, classification results beyond the rigidity region in the strictly pseudoconvex case have been obtained (\cite{HuangJi01}, \cite{Hamada05}, \cite{HuangJiXu06}, \cite{HuangJiYin09}, \cite{HuangJiYin12}, \cite{E13}; see also \cite{Faran82}). There is also an ongoing effort to understand the complexity of rational mappings (as measured by the degree) between spheres, asymptotically as the codimension of the mappings tend to infinity (see, e.g., \cite{DAngelo88}, \cite{Dangelo91}, \cite{DAngeloKR03}, \cite{JPDLebl09}, \cite{JPDLebl11}, \cite{LeblPeters11}). Common to these results is that the arguments are essentially local, meaning that the they rest mainly upon the fact that the mapping under consideration sends a neighborhood of a point on $M$ into the hyperquadric.
When the signatures of $M$ and the target hyperquadric are different, then it is easy to see that no rigidity can hold for local mappings, but some partial rigidity (meaning that that the image is contained in a projective subspace) persists when the signature difference is small (\cite{EHZ05}, \cite{BEH09}, \cite{ESh12}). In a recent paper by S.-C. Ng \cite{Ng13}, holomorphic mappings between hyperquadrics of different signatures are studied in a {\it global} setting, and it is shown that much stronger conclusions can then be drawn: Namely, a proper holomorphic mapping from the $\ell$-concave quadratic "flag domain" $\bB^{n+1}_\ell$ (see below for the precise definition) into the higher dimensional, $\ell'$-concave quadratic domain $\bB^{N+1}_{\ell'}$ is necessarily linear when $\ell'-\ell<\ell$. This result can also be obtained by a direct application in the global setting (see, for instance, \cite{Huang}) of the normal form obtained in the paper \cite{BEH09} by M. S. Baouendi, X. Huang, and the present author. The proof in \cite{Ng13} is not based on the local normal form in \cite{BEH09}; instead, it is based on global arguments resting on the classical Feder's Theorem, and these arguments seem to have other applications as well (that will not be address in this paper). Also, as the local result in  \cite{BEH09} is quite technical, Ng's proof can be regarded as a much simplified approach in the global setting. As a final remark here, we mention that in an even more recent paper \cite{KimZaitsev13}, S.-Y. Kim and D. Zaitsev studied partial rigidity of proper holomorphic mappings between bounded symmetric domains and obtained a normal form analogous, in the setting of proper maps between bounded symmetric domains, to that in \cite{BEH09}; see also \cite{KimZaitsev12} for local results in this setting.

The present paper was inspired by the result in \cite{Ng13} described above. We shall consider proper holomorphic maps from a more general class of domains into $\bB^{N+1}_{\ell'}$, and prove a result (Theorem \ref{thm:main}) guaranteeing that all possible such maps are rational of a fixed degree. We begin with some notation and definitions. For integers $n\geq 1$ and $0\leq \ell\leq n$, the $\ell$-concave
quadric domain $\bB^{n+1}_\ell\subset \bCP^{n+1}$ is defined in homogeneous coordinates $Z=[Z_0:\ldots:Z_{n+1}]$ by the equation
\begin{equation}\label{e:Bell}
-\sum_{i=0}^\ell |Z_i|^2+\sum_{i=\ell+1}^{n+1}|Z_i|^2<0.
\end{equation}
We note that for $\ell=0$, the $0$-concave quadric domain $\bB^{n+1}_0$ is just the standard unit ball $\bB^{n+1}$ in $\bC^{n+1}\subset \bCP^{n+1}$.
In this paper, we shall, however, always assume that $\ell>0$. Moreover, it is easy to see that for each $\ell>0$ there is an automorphism of $\bCP^{n+1}$
(reordering the coordinates) that maps $\bB^{n+1}_\ell$ to the complement of $\bB^{n+1}_{n-\ell}$. We shall further restrict our attention in this paper to
the case $\ell\leq n/2$; in other words, we shall assume $0<\ell\leq n/2$. The boundary of $\bB^{n+1}_\ell$ is then the standard Levi nondegenerate hyperquadric
of CR dimension $n$ and Levi signature $\ell$, here denoted by $Q^n_\ell$.

We shall consider a more general class of $\ell$-concave domains, defined as follows. Let $m\geq 1$ and $N_0\geq n$ be integers, and
$P_{\ell+1}(\zeta,Z_{\ell+1},\ldots,Z_{n+1}), \ldots, P_{N_0+1}(\zeta,Z_{\ell+1},\ldots,Z_{n+1})$ homogeneous polynomials of degree $m$ such that
\begin{equation}\label{e:P=0}
P_k(\zeta,0,\ldots,0)\equiv 0,\quad k=\ell+1,\ldots, N_0+1.
\end{equation}
Consider the domain
$\Omega\subset\bCP^{n+1}$ defined by
\begin{equation}\label{e:Pell}
-\sum_{i=0}^\ell |Z_i|^{2m}+\sum_{k=\ell+1}^{N_0+1}|P_k(Z_{0},Z_{\ell+1},\ldots,Z_{n+1})|^2<0.
\end{equation}
We note that $\Omega$ is the interior of the closure in $\bCP^{n+1}$ of the domain $\Omega_0\subset \bC^{n+1}$ defined in affine coordinates $z=(Z_1/Z_0,\ldots,Z_{n+1}/Z_0)$ by
\begin{equation}\label{e:Pell0}
-\sum_{i=1}^\ell |z_i|^{2m}+\sum_{k=\ell+1}^{N_0+1}|Q_k(z_{\ell+1},\ldots,z_{n+1})|^2<1,
\end{equation}
where
\begin{equation}\label{e:Q}
Q_k(z_{\ell+1},\ldots,z_{n+1}):=
P_k(1,z_{\ell+1},\ldots,z_{n+1}).
\end{equation}
We shall impose the following additional condition on the polynomials $P_k$ (via the affine polynomials $Q_k$ defined by \eqref{e:Q}):
\medskip

\noindent
{\bf (I)} The rank of the $(N_0+1-\ell)\times (n+1-\ell)$ Jacobian matrix
\begin{equation}\label{e:restJacP}
\left(\frac{\partial Q_k}{\partial z_j}\right)_{k=\ell+1,\ldots,N_0+1,\ j=\ell+1,\ldots, n+1}
\end{equation}
is $n+1-\ell$ at some point in $\bC^{n+1}$ (and hence outside an algebraic variety).

\medskip

\begin{example}\label{ex:ex1} We note that $\Omega=\bB^{n+1}_\ell$ when $N_0=n$, $m=1$, and $$P_k(\zeta,Z_{\ell+1},\ldots,Z_{n+1}):=Z_k.$$
Clearly, this satisfies \eqref{e:P=0} and {\bf (I)} above. A slightly more general example is obtained when we still choose $N_0=n$, leave $m\geq 1$ arbitrary, and choose $$P_k(\zeta,Z_{\ell+1},\ldots,Z_{n+1}):=\zeta^{m-p_k}Z_k^{p_k},$$ for some positive integers $1\leq p_k\leq m$. In affine coordinates, the domain $\Omega_0\subset \bC^{n+1}$ is then the "pseudohyperboloid" given by
\begin{equation}\label{e:psellipse}
-\sum_{i=1}^\ell |z_i|^{2m}+\sum_{k=\ell+1}^{n+1}|z_{k}|^{2p_k}<1,
\end{equation}
and, again, conditions \eqref{e:P=0} and {\bf (I)} are clearly satisfied.
\end{example}

\begin{remark} {\rm We should point out that if $N_0=n$, then condition {\bf (I)} implies that $\partial \Omega$ is {\it locally} biholomorphic to $Q^{n}_{\ell}$ outside a proper real-algebraic subvariety in $\partial\Omega$. However, we want to emphasize that the considerations in this paper are {\it global} and cannot be reduced to the local situation.
}
\end{remark}

The main result in this paper is the following.

\begin{theorem}\label{thm:main} Let $n\geq1$, $1\leq \ell\leq n/2$, $m\geq 1$ and $N_0\geq n$ be integers as above. For $k=\ell+1,\ldots,N_0+1$, let  $P_{k}(\zeta,Z_{\ell+1},\ldots,Z_{n+1})$ be homogeneous polynomials of degree $m$, and assume that \eqref{e:P=0} and {\bf (I)} above hold. Let $\Omega\subset\bCP^{n+1}$ be defined by \eqref{e:Pell} and $H\colon \Omega\to \bB^{N+1}_{\ell'}$, for some $N\geq n$ and $\ell\leq \ell'\leq N/2$, a proper holomorphic mapping. Then $H$ is rational, and
if
\begin{equation}\label{e:signdiff}
(\ell'-\ell)+(N_0-n)<\ell,
\end{equation}
then $\deg H=m$.
\end{theorem}

\begin{remark} {\rm As mentioned in Example \ref{ex:ex1} above, if we take $N_0=n$, and $$P_k(\zeta,Z_{\ell+1},\ldots,Z_{n+1})=Z_k$$ (i.e., $m=1$), then Theorem \ref{thm:main} states that any proper holomorphic mapping $H\colon \bB^{n+1}_{\ell}\to \bB^{N+1}_{\ell'}$, with $\ell'-\ell<\ell$, is rational of degree 1, i.e., {\it linear}; this is the result in the recent paper \cite{Ng13}. The author wishes to reiterate that the work reported on in this paper was inspired by \cite{Ng13}, and the proof that $H$ is rational in Theorem \ref{thm:main} is borrowed from that paper. As mentioned above, the main ideas in \cite{Ng13} and the present paper are different, though. The idea in \cite{Ng13} hinges on Feder's Theorem regarding rational maps between certain projective spaces being linear, whereas the idea in this paper is to use a normal form for local holomorphic mappings into hyperquadrics proved in \cite{EHZ05} (resulting in Proposition \ref{prop:main}). Thus, the proof given here also provides an alternative proof of the result in \cite{Ng13}.
}
\end{remark}

The proof of Theorem \ref{thm:main} has two main parts, (1) proving that $H$ is rational (Lemma \ref{lem:main} in Section \ref{s:mainlem} below; as mentioned above, the main idea in this proof is in \cite{Ng13}), and (2) proving Proposition \ref{prop:main} below from which the conclusion of Theorem \ref{thm:main} then follows. Before stating the proposition, we shall need to introduce some terminology. If $M\subset \bC^{n+1}$ is a real-analytic Levi nondegenerate hypersurface of Levi signature $\ell\leq n/2$, then for every $p\in$ there are local holomorphic coordinates $(z,w)=(z_1,\ldots,z_n,w)\in \bC^n\times \bC$ vanishing at $p$ such that in these coordinates $M$ is given by an equation of the form
\begin{equation}\label{e:Imw}
\im w=-\sum_{i=1}^\ell |z_i|^2+\sum_{i=\ell+1}^n|z_i|^2+A(z,\bar z,\re w),
\end{equation}
where $A(z,\bar z,\re w)$ is a real-analytic function that vanishes to weighted order ($z$ has order 1 and $w$ has order 2) at least 4. A domain $\Omega$ with $M$ as part of its boundary is said to be $\ell$-concave with respect to $M$ if every point $p\in M$ has an open neighborhood $U$ such that in the coordinates $(z,w)$ it holds that $\Omega\cap U$ is contained
\begin{equation}\label{e:Imw>}
\im w>-\sum_{i=1}^\ell |z_i|^2+\sum_{i=\ell+1}^n|z_i|^2+A(z,\bar z,\re w).
\end{equation}
This is equivalent to saying that if we use $\theta:=i\partial\rho$, with
\begin{equation}\label{e:Imwrho}
\rho:=-\im w-\sum_{i=1}^\ell |z_i|^2+\sum_{i=\ell+1}^n|z_i|^2+A(z,\bar z,\re w),
\end{equation}
as a contact form on $M$, then the Levi form with respect to $\theta$ has $\ell$ negative and $n-\ell$ positive eigenvalues. (We note that if $\ell=n/2$, then both sides of $M$ are $\ell$-concave with respect to $M$.)

We are now ready to state the main new technical result in this paper. The result summarizes the properties of $\Omega$ that are required for the conclusion of Theorem \ref{thm:main} to hold, once the rationality of $H$ has been established. This result should also be of independent interest as it can be applied in other situations as well.

\begin{proposition}\label{prop:main} Let $\Omega\subset\bCP^{n+1}$ be a domain satisfying the following three conditions:
\smallskip

\noindent{\rm (i)} There exists an $\ell$-dimensional algebraic variety $X$ contained in $\Omega$.
\smallskip

\noindent{\rm (ii)} A relatively open piece of the boundary $\partial\Omega$ is a real-analytic Levi nondegenerate hypersurface $M$ of signature $\ell$, and $\Omega$ is $\ell$-concave with respect to $M$.
\smallskip

\noindent{\rm (iii)} There exists an open neighborhood $U$ of $M$ and a nonconstant holomorphic mapping $H^0\colon U\subset \bCP^{N_0+1}$ such that $H^0(M)\subset Q^{N_0}_\ell$.
\smallskip

\noindent
If $H\colon \Omega\to \bB^{N+1}_{\ell'}$ is a proper holomorphic mapping that extends as a rational mapping $\bCP^{n+1}\to \bCP^{N+1}$, and $(\ell'-\ell)+(N_0-n)<\ell$, then $H_0$ is also rational and $\deg H=\deg H_0$.
\end{proposition}

\begin{remark}\label{rem:mu} {\rm The minimum of $N_0-n$ for which there is a mapping $H^0$ as in Proposition \ref{prop:main} is sometimes referred to as the CR complexity of $M$, often denoted $\mu(M)$.
}
\end{remark}

The proof of this proposition, as well as the proof of Theorem \ref{thm:main}, is given below in Section \ref{s:proofs}.

\section{Two Lemmas}\label{s:mainlem}

In this section, we shall prove that the mapping $H$ in Theorem \ref{thm:main} is rational, and also that $\Omega$ in this theorem satisfies the condition (ii) in Proposition \ref{prop:main}. We shall formulate these results as two separate lemmas. We first prove:

\begin{lemma}\label{lem:main} Let $n\geq1$, $1\leq \ell\leq n/2$, $m\geq 1$ and $N_0\geq n$ be integers. For $k=\ell+1,\ldots,N_0+1$, let  $P_{k}(\zeta,Z_{\ell+1},\ldots,Z_{n+1})$ be homogeneous polynomials of degree $m$, and assume that \eqref{e:P=0} and {\bf (I)} above hold. Let $\Omega\subset\bCP^{n+1}$ be defined by \eqref{e:Pell} and $H\colon \Omega\to \bB^{N+1}_{\ell'}$, for some $N\geq n$ and $\ell\leq \ell'\leq N/2$, a proper holomorphic mapping. Then $H$ is rational.
\end{lemma}

\begin{proof} As mentioned above, this proof is directly adapted from \cite{Ng13}. Let us denote by $U_k\subset\bCP^{n+1}$, for $k=0,\ldots, n+1$, the coordinate chart where the homogeneous coordinate $Z_k\neq 0$. We note that if $p_0=[Z^0_0:\ldots:Z^0_{n+1}]$ belongs to the complement $\Omega^c$ of $\Omega$, then
\begin{equation}
\sum_{i=0}^\ell |Z^0_i|^{2m}\leq\sum_{k=\ell+1}^{N_0+1}|P_k(Z^0_{0},Z^0_{\ell+1},
\ldots,Z^0_{n+1})|^2,
\end{equation}
and, hence, by \eqref{e:P=0} we must have $Z^0_{k}\neq 0$ for some $k=\ell+1,\ldots, n+1$. In particular, $p_0\in U_{k}$, and we conclude that
\begin{equation}\label{e:cover}
\Omega^c\subset\bigcup_{k=\ell+1}^{n+1} U_k.
\end{equation}
For each fixed $k=\ell+1,\ldots n+1$, let $z=(z_1,\ldots,z_{n+1})$, with $z_j=Z_{j-1}/Z_k$ for $j<k$ and $z_j=Z_{j}/Z_k$ for $j>k$, be affine coordinates in $U_k\cong \bC^{n+1}$, and note that the holomorphic mapping $H\colon \Omega\to \bB^{N+1}_{\ell'}$ can be expressed in $\Omega\cap U_k$ as $H^k(z)=[H^k_0(z):\ldots:H^k_{N+1}(z)]$, where the $H^k_j$ are holomorphic functions (without common zeros) in $\Omega\cap U_k$. We shall prove that, for each such $k$, every holomorphic function $f(z)$ in $\Omega\cap U_k$ can be extended as a holomorphic function in all of $U_k\cong \bC^{n+1}$. This implies that each $H^k$ extends as a rational map on $U_k$, and the obvious fact that $\Omega\cap U_k\cap U_{k'}$ is nonempty, for $k,k'\in \{\ell+1,\ldots, n+1\}$, combined with \eqref{e:cover}, then implies that $H$ extends to $\bCP^{n+1}$ as a well defined rational mapping, as claimed in Lemma \ref{lem:main}.

Let us fix $k_0\in \{\ell+1,\ldots,n+1\}$ and let $f(z)$ be a holomorphic function in $\Omega\cap U_{k_0}$, given in the affine coordinates $z=(z_1,\ldots,z_{n+1})\in \bC^{n+1}$ described above by
\begin{equation}\label{e:OmegaUk}
-\sum_{i=1}^{\ell+1} |z_i|^{2m}+\sum_{k=\ell+1}^{N_0+1}|R_k(z_{1},z_{\ell+2},\ldots,z_{n+1})|^2<0,
\end{equation}
where
\begin{equation}\label{e:R}
R_k(\zeta,z_{\ell+2},\ldots,z_{n+1}):=P_k(\zeta,z_{\ell+2},\ldots, z_{k_0},1,z_{k_0+1},\ldots,z_{n+1}).
\end{equation}
Here, the $1$ that appears as an argument in $P_k$ is in the place where $Z_{k_0}$ should appear. (Recall also that there is a shift of indices below $k_0$.) Let us decompose the coordinates as follows $z=(\zeta,\tau)\in \bC^{\ell+1}\times\bC^{n-\ell}$ and then rewrite \eqref{e:OmegaUk} as follows
\begin{equation}\label{e:OmegaUk2}
\sum_{i=1}^{\ell+1} |\zeta_i|^{2m}-\sum_{k=\ell+1}^{N_0+1}|R_k(\zeta_{1},\tau)|^2>0,
\end{equation}
Let us fix $\tau=\tau_0\in \bC^{n-\ell}$ and note that for this fixed value the slice $\Omega_{\tau_0}$ of $\Omega\cap U_{k_0}$ in $\bC^{\ell+1}\cong \{(\zeta,\tau)\colon \tau=\tau_0\}$ is given by an equation of the form
\begin{equation}\label{e:OmegaUkslice}
|\zeta_1|^{2m}-\sum_{k=\ell+1}^{N_0+1}|S_k(\zeta_{1})|^2+\sum_{i=2}^{\ell+1} |\zeta_i|^{2m}>0,
\end{equation}
where the single variable polynomials $S_k$ are given by $S_k(t):=R_k(t,\tau_0)$. Observe that by condition \eqref{e:P=0} each $S_k$ has degree at most $m-1$. It follows that the complement of $\Omega_{\tau_0}$ in $\bC^{\ell+1}$ is contained in a large polydisk $\{\zeta\colon |\zeta_i|<r_i,\ i=1,\ldots, \ell+1\}$ and, hence, compact. Moreover, we claim that $\Omega_{\tau_0}$ is connected. To see this, we decompose $\zeta=(x,y)\in \bC\times\bC^\ell$ and note that for fixed $x=x_0$ the $\ell$-dimensional slice in $\bC^\ell\cong \{(x,y)\colon x=x_0\}$ of $\Omega_{\tau_0}$ is of the form
\begin{equation}\label{e:lemncomp}
\sum_{j=1}^\ell|y_j|^{2m}>C_0,
\end{equation}
where $C_0$ is a constant (possibly negative, which means that the slice is all of $\bC$). Clearly this set is connected and non-empty for every $C_0$. Since $C_0$ moreover varies continuously with $x_0$, we deduce that $\Omega_{\tau_0}$ is connected. Since $f_{\tau_0}:=f(\cdot,\tau_0)$ is holomorphic in $\Omega_{\tau_0}$, we conclude by Hartogs' theorem (see Theorem 2.3.2 in \cite{Hormander90}) that $f_{\tau_0}$ extends holomorphically to $\bC^{\ell+1}$ (recall $\ell\geq 1$). Next, we observe that if $V$ is any open, bounded subset in $\bC^{n-\ell}$, then the coefficients of $S_k$ as functions of $\tau_0$ as $\tau_0$ varies over $V$ are uniformly bounded and, hence, there is an $r>0$ (depending on the bounded set $V$) such that $\Omega$ contains the open subset $\{\zeta\colon |\zeta_i|>r,\ i=1,\ldots, \ell+1\}\times V$. Since $f_{\tau_0}=f(\cdot,\tau_0)$ extends holomorphically to $\bC^{\ell+1}$ for each fixed $\tau_0$, we conclude  (cf. Lemma 2.2.11 in \cite{Hormander90}) that $f(\zeta,\tau)$ extends holomorphically to $\bC^{\ell+1}\times V$. Since $V$ is an arbitrary bounded open subset, we conclude that $f(\zeta,\tau)$ extends holomorphically to $\bC^{\ell+1}\times\bC^{n-\ell}$, which proves the claim above and, therefore, concludes the proof of Lemma \ref{lem:main}.
\end{proof}

We conclude this section by proving the following lemma:

\begin{lemma}\label{lem:M} Let $n\geq1$, $1\leq \ell\leq n/2$, $m\geq 1$ and $N_0\geq n$ be integers. For $k=\ell+1,\ldots,N_0+1$, let  $P_{k}(\zeta,Z_{\ell+1},\ldots,Z_{n+1})$ be homogeneous polynomials of degree $m$, and assume that \eqref{e:P=0} and {\bf (I)} above hold. If $\Omega\subset\bCP^{n+1}$ is defined by \eqref{e:Pell}, then a relatively open subset of the boundary $\partial\Omega$ is a real-algebraic, Levi nondegenerate hypersurface of Levi signature $\ell$. Moreover, $\Omega$ is $\ell$-concave with respect to $M$.
\end{lemma}

\begin{proof} It is easy to see $\Omega$ is non-empty; indeed, the $\ell$-dimensional projective plane $X$ defined by $Z_{\ell+1}=\ldots Z_{n+1}=0$ is contained in $\Omega$ in view of \eqref{e:P=0}. Similarly, it is easy to see that the complement $\Omega^c$ has non-empty interior as well. Thus, the algebraic set $\partial\Omega$, defined by
\begin{equation}\label{e:Pell=0}
-\sum_{i=0}^\ell |Z_i|^{2m}+\sum_{k=\ell+1}^{N_0+1}|P_k(Z_{0},Z_{\ell+1},\ldots,Z_{n+1})|^2=0,
\end{equation}
has a component of real codimension $1$, and therefore a relatively open piece of it is a real-algebraic hypersurface (i.e., nonsingular and of real codimension 1). It remains to show that this hypersurface, denoted $\hat M$, is Levi nondegenerate and of signature $\ell$ outside a proper subvariety $\hat S\subset \hat M$, and that $\Omega$ is $\ell$-concave with respect to $M:=\hat M\setminus \hat S$. Recall that in affine coordinates $z=(z_1,\ldots, z_{n+1})$ in $\bC^{n+1}$ ($=U_0=\{Z_0\neq 0\}$), the hypersurface $M$ is defined by
\begin{equation}\label{e:Pell0=0}
\rho(z,\bar z):=-\sum_{i=1}^\ell |z_i|^{2m}+\sum_{k=\ell+1}^{N_0+1}|Q_k(z_{\ell+1},\ldots,z_{n+1})|^2-1=0,
\end{equation}
where the $Q_k$ are given by \eqref{e:Q}, and $\Omega_0:=\Omega\cap \bC^{n+1}$ is given by $\rho<0$. (More precisely, $M$ is a relatively open subset of the set defined by \eqref{e:Pell0=0}.) We shall show that outside a thin subset of $M$ (henceforth meaning a subvariety of positive codimension in $M$) the Levi form defined by the contact form $\theta:=i\partial \rho$ has $\ell$ negative and $n-\ell$ positive eigenvalues. This will complete the proof of Lemma \ref{lem:M}.

Let us decompose the coordinates in $\bC^{n+1}=\bC^\ell\times \bC^{n+1-\ell}$ as  $z=(\zeta,\tau)$, with $\zeta=(\zeta_1,\ldots,\zeta_\ell)$ and $\tau=(\tau_1,\ldots,\tau_{n+1-\ell})$. For fixed $\zeta^0\in \bC^\ell$, the slice $M_{\zeta^0}:=M\cap\{(\zeta,\tau)\colon \zeta=\zeta^0\}$ is given by an equation of the form
\begin{equation}\label{e:zetaslice}
\sum_{k=\ell+1}^{N_0+1}|Q_k(\tau)|^2=1+C_0,
\end{equation}
where $C_0\geq 0$. Note that the polynomial mapping
$$
\tau\mapsto (Q_{\ell+1}(\tau),\ldots, Q_{N_0+1}(\tau))
$$
sends the slice $M_{\zeta^0}$ into the sphere of radius $\sqrt{1+C_0}$ in $\bC^{N_0+1-\ell}$, and the slice $\Omega_{\zeta^0}:=\Omega_0\cap\{(\zeta,\tau)\colon \zeta=\zeta^0\}$ into the corresponding ball. At a point $p\in M_{\zeta^0}$ where this mapping has full rank (which is outside a thin subset by condition {\bf (I)}), standard Levi form considerations show that $M_{\zeta^0}$ is strictly pseudoconvex at $p$, and we conclude that the Levi form with respect to $i\partial \rho$, restricted to the slice $\{(\zeta,\tau)\colon \zeta=\zeta^0\}$, has $n-\ell$ positive eigenvalues. It follows that the Levi form of $M$ with respect to $\theta=i\partial \rho$ has at least $n-\ell$ positive eigenvalues outside a thin subset.

Next, we choose a linear map $L\colon (\bC,0)\to (\bC^{n+1-\ell},0)$ such that the function
\begin{equation}\label{e:q}
q(w,\bar w):=\sum_{k=\ell+1}^{N_0+1}|Q_k(L(w))|^2-1
\end{equation}
is not constant. Consider the pullbacks $M_L$ and $\Omega_L$ to $V\times\bC$, where $V$ is a simply connected open subset of $(\bC^*)^\ell$ with $\bC^*:=\bC\setminus \{0\}$, of $M$ and $\Omega_0$ via the map
\begin{equation}\label{e:linmap}
(\zeta,w)\mapsto (\zeta_1^{1/m},\ldots \zeta_\ell^{1/m},L(w)).
 \end{equation}
The hypersurface $M_L$ is then given by the equation
\begin{equation}\label{e:M_L}
-\rho_L:=\sum_{i=1}^{\ell}|\zeta_i|^2-q(w,\bar w)=0,
\end{equation}
and $\Omega_L$ is given by $\rho_L<0$. It is well known, and easy to see, that a hypersurface defined by \eqref{e:M_L} is strictly pseudoconvex at a point $(\zeta^0,w_0)$, where $q(w_0,\bar w_0)>0$, if and only if the function $-\log q(w,\bar w)$ is strictly subharmonic near $w_0$. Since $q$ is given by \eqref{e:q}, we conclude (cf.\ Corollary 1.6.8 in \cite{Hormander90}) that $M_L$ is strictly pseudoconvex outside a thin set. Clearly, the Levi form of $M_L$ then has $\ell$ positive eigenvalues with respect to $i\partial(-\rho_L)$, and therefore the Levi form with respect to $i\partial(\rho_L)$ has $\ell$ negative eigenvalues. We conclude that the Levi form of $M$ with respect to $i\partial \rho$ has at least $\ell$ negative eigenvalues outside a thin set. Consequently, outside a thin set the Levi form of $M$ with respect to $i\partial\rho$ has at least $\ell$ negative and at least $n-\ell$ positive eigenvalues. Since the CR dimension of $M$ is $n=\ell+(n-\ell)$, we conclude that, outside a thin set, the Levi form with respect to $i\partial\rho$ has precisely $\ell$ negative and $n-\ell$ positive eigenvalues. This completes the proof of Lemma \ref{lem:M}.
\end{proof}

\section{Proofs of Proposition \ref{prop:main} and Theorem \ref{thm:main}}\label{s:proofs}

In this section, we shall prove Proposition \ref{prop:main} and Theorem \ref{thm:main}.

\begin{proof}[Proof of Proposition $\ref{prop:main}$] For simplicity of notation, we shall use the notation $H$ also for the rational mapping $\bCP^{n+1}\to \bCP^{N+1}$ that extends the proper mapping $\Omega\to \bB^{N+1}_{\ell'}$. After shrinking $M$ and $U$ if necessary, we may assume that $M$ and $U$ are contained in $\bC^{n+1}:=\{[Z_0:\ldots Z_{n+1}]\in \bCP^{n+1}\colon Z_0\neq 0\}$, and that $H|_U$ is a holomorphic mapping $U\to \bC^{N+1}$. Since $H\colon \Omega\to \bB^{N+1}$ is proper, we also have $H|_U(M)\subset Q^N_{\ell'}\cap \bC^{N+1}$. We claim that $H|_U$ is CR transversal (which is the same as transversal in the real sense in this case, since the CR manifolds are hypersurfaces) to $M':=Q^N_{\ell'}\cap \bC^{N+1}$ outside a proper real-analytic subset of $M$. Indeed, this follows directly from Theorem 1.1 in \cite{BER07}: Since $(\ell'-\ell)+(N_0-n)<\ell$ and $\ell\leq n/2$, it is clear that $\ell'\leq n-1$ and therefore (1.2) of this theorem holds (the dictionary between the present paper and \cite{BER07} is that $e(M',p')=\ell'$ and $e_0(M',p')=0$). Moreover, since $H|_U(\Omega\cap U)\subset \bB^{N+1}_{\ell'}$ by assumption, the possibility (ii) in this theorem is excluded. Thus, the conclusion is that (i) of Theorem 1.1 in \cite{BER07} holds, which is precisely the claim above. By again shrinking $M$ and $U$ if necessary, we may assume that $H|_U$ is CR transversal to $Q^N_{\ell'}$ along all of $M$. It now follows from Theorem 1.2(a) in \cite{ESh12} that $H|_U(U)$ is contained in complex plane of dimension $d+1:=N_0+(\ell'-\ell)+1$; indeed, by assumption, $H|_U$ is "side preserving" in the sense of that paper (or $\ell=n/2$), and assumption (i) in Theorem 1.2(a) is also assumed in Proposition \ref{prop:main}. It then also follows that the rational mapping $H$ sends $\bCP^{n+1}$ into a projective plane $\Pi$ of dimension $d+1$ ($=N_0+(\ell'-\ell)+1$). By elementary linear algebra, the Hermitian quadratic form
\begin{equation}\label{e:qform}
-\sum_{i=0}^{\ell'} |Z_i|^2+\sum_{i=\ell'+1}^{N+1}|Z_i|^2
\end{equation}
restricted to a projective plane of dimension $d+1$ is a Hermitian quadratic form with $\ell''+1$ negative eigenvalues, $r$ zero eigenvalues, and $p+1:=d-\ell''-r+1$ positive eigenvalues, where
\begin{equation}\label{e:signest}
\ell''\leq \ell', \quad p\leq N-\ell', \quad r\leq\min(\ell'-\ell'', N-\ell'-p).
\end{equation}
Consequently, we may choose homogeneous coordinates $[\zeta_0:\ldots:\zeta_{d+1}]$ on $\Pi\cong\bCP^{d+1}$ such that the restriction of \eqref{e:qform} to this projective space is given by
\begin{equation}\label{e:degqform}
-\sum_{i=0}^{\ell''}|\zeta_i|^2+\sum_{i=\ell''+r+1}^{d+1}|\zeta_i|^2.
\end{equation}
We now embed $\Pi$ into $\bCP^{d+r+1}$, with coordinates $[Z'_0:\ldots:Z'_{d+r+1}]$, as the projective plane $Z_i=\zeta_i$, for $i=0,\ldots,d+1$, and $Z'_{d+1+j}=\zeta_{\ell''+j}$ for $j=1,\ldots,r$. We observe that the quadratic form \eqref{e:degqform} is then the restriction to $\Pi$ of the {\it nondegenerate} Hermitian form
\begin{equation}\label{e:nondegqform}
-\sum_{i=0}^{\ell''+r} |Z'_i|^2+\sum_{i=\ell''+r+1}^{d+r+1}|Z'_i|^2.
\end{equation}
We note, by \eqref{e:signest}, that
\begin{equation}\label{e:signest2}
\ell''+r\leq \ell',\quad d+r\leq N_0+2(\ell'-\ell).
 \end{equation}
We also note, however, that there is a possibility that $\ell''+r> d-\ell''$, i.e., the number of negative terms in \eqref{e:nondegqform} exceed the number of positive terms. Composing the rational mapping $H$ with the embedding $\Pi\to\bCP^{d+r+1}$ above yields a rational mapping $H'$, clearly with $\deg H'=\deg H$, such that $H'\colon \Omega\to \bB^{d+r+1}_{\ell''+r}$ is a proper holomorphic mapping and $H'$ is transversal to $Q^{d+r+1}_{\ell''+r}$ along $M$. By shrinking $M$ and $U$ again, if necessary, we may assume also that $H^0\colon U\to \bCP^{N_0+1}$ is also transversal to $Q^{N_0}_{\ell}$ along $M$. Let us pick a point $p\in M$ and apply automorphisms $T'_1$ and $T^0_1$ of $\bCP^{d+r+1}$ and $\bCP^{N_0+1}$, respectively, such that $H'(p_0)=0$, $H^0(p_0)=0$, and such that $Q^{d+r}_{\ell''+r}$ and $Q^{N_0}_\ell$ are mapped to $\bH^{d+r}_{\ell''+r}$ and $\bH^{N_0}_\ell$, respectively, where $\bH^s_q$ is the "Heisenberg realization" of $\bB^{s}_q$, given in affine coordinates $(z,w)=(z_1,\ldots,z_s,w)\in \bC^s\times\bC$ by
\begin{equation}\label{e:heisen}
\im w=-\sum_{i=1}^q|z_i|^2+\sum_{i=q+1}^{s}|z_i|^2.
\end{equation}
We can express the two maps $T'_1\circ H'$ and $T^0_1\circ H^0$ as
$$
T'_1\circ H'=(F',G')=(F'_1,\ldots, F'_{d+r},G'),\quad T^0_1\circ H^0=(F^0,G^0)=(F^0_1,\ldots, F^0_{N_0},G^0).
$$
Let us now choose local holomorphic coordinates $(z,w)\in \bC^n\times \bC$, vanishing at $p\in M$, such that $M$ is given by \eqref{e:Imw}, and introduce
\begin{equation}\label{e:dgdw}
\sigma':=\frac{\partial G'}{\partial w}(0,0),\quad
\sigma^0:=\frac{\partial G^0}{\partial w}(0,0).
\end{equation}
It is well known that $\sigma'$ and $\sigma^0$ are real, and that transversality is equivalent to these quantities being nonzero. We may assume, without loss of generality, that $\sigma^0>0$; if $\ell<n/2$, this is necessarily the case, and if $\ell=n/2$ we may replace $w$ by $-w$ to achieve this.

We shall first consider the case where also $\sigma'>0$. (This is necessarily the case if, e.g., $\ell'<n-\ell$; see, e.g., \cite{EHZ05}.) Standard Levi form considerations (e.g., examining the quadratic terms in the equation stating that $(F',G')$ maps $M$ into $\bH^{d+r}_{\ell''+r}$) imply that $\ell\leq \ell''+r$ and $n-\ell\leq d-\ell''$. Let us denote the codimensions of the maps $T'_1\circ H'$ and $T^0_1\circ H^0$ by $k':=d+r-n$ and $k^0:=N_0-n$, respectively. We have, by \eqref{e:signest2} and the assumption $(N_0-n)+(\ell'-\ell)<\ell\leq n/2$ in Proposition \ref{prop:main},
$$
k'+k^0\leq N_0+2(\ell'-\ell)-n+N_0-n=2\left((\ell'-\ell)+(N_0-n)\right)<2\ell\leq n.
$$
It now follows from Theorem 5.3 in \cite{EHZ05} (see also the discussion in the paragraph preceding this theorem; the assumption $\ell_q<n-\ell$ in the theorem is simply to guarantee that $\sigma'\sigma^0=1$, which is already assumed here) that there are formal holomorphic coordinates $(z,w)\in \bC^n\times\bC$, automorphisms $T'_2$ and $T^0_2$ of $\bH^{d+r}_{\ell''+r}$ and $\bH^{N_0}_\ell$, respectively, such that in these coordinates we have
\begin{equation}\label{e:normform1}
(T'_2\circ T'_1\circ H')(z,w)=(z,\phi'(z,w),w),\quad (T^0_2\circ T^0_1\circ H^0)(z,w)=(z,\phi^0(z,w),w),
\end{equation}
where $\phi'=(\phi'_1,\ldots,\phi'_{k'})$ and $\phi^0=(\phi^0_1,\ldots,\phi^0_{k^0})$ are formal power series vanishing to at least second order at $0$. Moreover, we have (see Remark 5.4 following Theorem 5.3)
\begin{equation}\label{e:normform2}
-\sum_{i=1}^{\ell''+r-\ell}|\phi'_i|^2+\sum_{i=\ell''+r-\ell+1}^{k'}|\phi'_i|^2=
\sum_{i=1}^{k^0}|\phi^0_i|^2.
\end{equation}
If we write $\phi'=(\psi, \phi)$ with $\psi=(\phi'_1,\ldots,\phi'_{\ell''+r-\ell})$ and $\phi=(\phi'_{\ell''+r-\ell +1},\ldots, \phi'_{k'})$, then (by also substituting $k':=d+r-n$ and $k^0:=N_0-n$) equation \eqref{e:normform2} can be rewritten as
\begin{equation}\label{e:normform3}
\sum_{i=1}^{(N_0-n)+(\ell'-\ell'')}|\phi_i|^2=\sum_{i=1}^{\ell''+r-\ell}|\psi_i|^2+\sum_{i=1}^{N_0-n}|\phi^0_i|^2,
\end{equation}
which implies that there is a unitary matrix $U$ of the appropriate size such that either
\begin{equation}\label{e:normform1c1}
\phi =(\psi,\phi^0,0,\ldots,0)U,\quad\text{{\rm if $(N_0-n)+(\ell'-\ell'')\geq(N_0-n)+(\ell''+r-\ell)$}},
\end{equation}
or
\begin{equation}\label{e:normform1c2}
(\phi,0\ldots, 0) =(\psi,\phi^0)U,\quad\text{{\rm if $(N_0-n)+(\ell'-\ell'')<(N_0-n)+(\ell''+r-\ell)$}},
\end{equation}
where zero components are added so that the vectors on each side have the same number of components (and, of course, in the case $(N_0-n)+(\ell'-\ell'')=(N_0-n)+(\ell''+r-\ell)$ no zeros need to be added at all). If we now return to an original system of affine coordinates $(u_1,\ldots,u_{n+1})$, with $u_j=Z_j/Z_0$, in $\bC^{n+1}$, we obtain
\begin{equation}\label{e:unravel1}
(T'_2\circ T'_1\circ H')(u)=(f(u),\psi(u),\phi(u),g(u)),\quad (T^0_2\circ T^0_1\circ H^0)(u)=(f(u),\phi^0(u),g(u)),
\end{equation}
where $f(u)=z(u)$, $g(u)=w(u)$, and where, by a slight abuse of notation, we have used $\psi(u)$, $\phi(u)$, $\phi^0(u)$ for $$\psi(z(u),w(u)),\quad  \phi(z(u),w(u)), \quad \phi^0(z(u),w(u)).
$$
Since $T'_2\circ T'_1\circ H'$ is rational it is clear from \eqref{e:normform1c1}, \eqref{e:normform1c2}, and \eqref{e:unravel1} that so is $T^0_2\circ T^0_1\circ H^0$ and, therefore, $H^0$, as claimed in Proposition \ref{prop:main}. It remains to show that $\deg H'=\deg H^0$, which of course is equivalent to $\deg T'_2\circ T'_1\circ H'= \deg T^0_2\circ T^0_1\circ H^0$. Let us write
\begin{equation}\label{e:rat0}
(T^0_2\circ T^0_1\circ H^0)(u)=\frac{(\hat f(u),\hat \phi^0(u),\hat g(u))}{q(u)},
\end{equation}
where $q(u)$, $\hat g(u)$, and the components of $\hat f(u)$, $\hat \phi^0(u)$  are polynomials in $u$, without common (nontrivial) factors. It follows from \eqref{e:normform1c1}, \eqref{e:normform1c2}, and \eqref{e:unravel1} that there is a matrix $A$ such that
\begin{equation}\label{e:rat'}
(T'_2\circ T'_1\circ H')(u)=\left (\frac{\hat f(u)}{q(u)},\psi(u),\left (\psi(u),\frac{\hat \phi^0(u)}{q(u)}\right )A,\frac{\hat g(u)}{q(u)}\right ),
\end{equation}
Let us write $\psi(u)=\hat \psi(u)/p(u)$, where $p(u)$ and the components of $\hat \psi(u)$ are polynomials without common factors. We claim that $p$ divides $q$. Indeed, suppose not. Then we can write $p(u)=r_1(u)s(u)$ and $q(u)=r_2(u)s(u)$, where the polynomial $r_1$ has degree at least 1, and  $r_1$ and $r_2$ are relatively prime. We then have
\begin{equation}\label{e:rat'1}
(T'_2\circ T'_1\circ H')(u)=\frac{\left (\hat f(u)r_1(u),\hat\psi(u)r_2(u),\left (\hat \psi(u)r_2(u),\hat \phi^0(u)r_1(u)\right )A,\hat g(u)r_1(u)\right )}{r_1(u)r_2(u)s(u)},
\end{equation}
where the components in the numerator and the denominator have no common factors. We note that the set of indeterminacy in $\bCP^{n+1}$ of this map contains a set that is defined by the vanishing of $\ell''+r-\ell+1$ polynomials, namely $r_1$ and the components of $\hat\psi$. Recall that $\Omega$ contains an algebraic variety $X$ of dimension $\ell$. Since $\ell''+r\leq \ell'$ and $\ell'-\ell<\ell$, we conclude that the set of indeterminacy of $T'_2\circ T'_1\circ H'$ meets $X$ and therefore this map is not holomorphic in $\Omega$. This is a contradiction, and thus $p$ divides $q$. Let us write $q(u)=r(u)p(u)$ (where of course $r(u)$ could be constant), and obtain
\begin{equation}\label{e:rat'2}
(T'_2\circ T'_1\circ H')(u)=\frac{\left (\hat f(u),\hat\psi(u)r(u),\left (\hat \psi(u)r(u),\hat \phi^0(u)\right )A,\hat g(u)\right )}{q(u)}.
\end{equation}
We now claim that
$$\deg \hat\psi r\leq \max (\deg \hat f, \deg \hat \phi^0, \deg \hat g, \deg q),$$
 which would prove that $\deg H'\leq\deg H^0$. Suppose, in order to reach a contradiction, that one component of $\hat \psi r$ has degree $m+1$, $\deg \hat \psi r\leq m+1$, and $\max (\deg \hat f, \deg \hat \phi^0, \deg \hat g, \deg q)\leq m$. Let us express the rational mapping $T'_2\circ T'_1\circ H'$ in homogeneous coordinates $Z=[Z_0:\ldots:Z_{n+1}]$ on the source space $\bCP^{n+1}$ and $Z'=[Z'_0:\ldots:Z'_{d+r+1}]$ on the target space,
 \begin{multline}\label{e:rat'3}
(T'_2\circ T'_1\circ H')(Z)=\\
\bigg [Z_0^{m+1}q\left(\frac{\hat Z}{Z_0}\right):Z_0^{m+1}\hat f\left(\frac{\hat Z}{Z_0}\right):Z_0^{m+1}(\hat\psi r)\left(\frac{\hat Z}{Z_0}\right):\\
\left (Z_0^{m+1}\hat \psi r\left(\frac{\hat Z}{Z_0}\right),Z_0^{m+1}\hat \phi^0\left(\frac{\hat Z}{Z_0}\right)\right )A:Z_0^{m+1}\hat g\left(\frac{\hat Z}{Z_0}\right)\bigg ],
\end{multline}
where $\hat Z=(Z_1,\ldots, Z_{n+1})$. Since $\max (\deg \hat f, \deg \hat \phi^0, \deg \hat g, \deg q)\leq m$, the set of indeterminacy in homogeneous coordinates contains the set where $Z_0=0$ and
$$
Z_0^{m+1}(\hat\psi r)\left(\frac{\hat Z}{Z_0}\right)=0,
$$
which again has codimension at most $\ell''+r-\ell+1$ in $\bCP^{n+1}$ and, therefore, must meet the $\ell$-dimensional variety $X\subset \Omega$. This is a contradiction, and we conclude that $\deg H'\leq\deg H^0$. To show equality, we note that it is clear from \eqref{e:normform3} that if we also write $\phi(u)=\hat \phi(u)/p(u)=\hat\phi(u)r(u)/q(u)$, then $\max(\deg \hat\psi r,\deg \hat \phi r)\geq\max (\deg\hat \phi^0)$. We deduce that $\deg H'=\deg H^0$.

To complete the proof of Proposition \ref{prop:main} in general, we must also consider the case where $\sigma'<0$. Levi form considerations as above shows that $\ell\leq d-\ell''$ and $n-\ell\leq \ell''+r$. The same argument that leads to Theorem 5.3 in \cite{EHZ05} (i.e., the proofs of Theorems 1.6 and 1.7) in this case shows that there are formal holomorphic coordinates $(z,w)\in \bC^n\times\bC$, automorphisms $T'_2$ and $T^0_2$ of $\bH^{d+r}_{\ell''+r}$ and $\bH^{N_0}_\ell$, respectively, such that in these coordinates we have
\begin{equation}\label{e:normformsig1}
\begin{aligned}
(T'_2\circ T'_1\circ H')(z,w) &=(z_{\ell+1},\ldots,z_{n},z_1,\ldots,z_\ell,\phi'(z,w),-w),\\ (T^0_2\circ T^0_1\circ H^0)(z,w) &=(z,\phi^0(z,w),w),
\end{aligned}
\end{equation}
where as above $\phi'=(\phi'_1,\ldots,\phi'_{k'})$ and $\phi^0=(\phi^0_1,\ldots,\phi^0_{k^0})$ are formal power series vanishing to at least second order at $0$. In this case, we have (see (5.10) in \cite{EHZ05})
\begin{equation}\label{e:normformsig2}
-\left(-\sum_{i=1}^{\ell''+r-\ell}|\phi'_i|^2+\sum_{i=\ell''+r-\ell+1}^{k'}
|\phi'_i|^2\right)=\sum_{i=1}^{k^0}|\phi^0_i|^2.
\end{equation}
If we again write $\phi'=(\psi, \phi)$ with $\psi=(\phi'_1,\ldots,\phi'_{\ell''+r-\ell})$ and $\phi=(\phi'_{\ell''+r-\ell +1},\ldots, \phi'_{k'})$, then \eqref{e:normformsig2} can be rewritten as
\begin{equation}\label{e:normformsig3}
\sum_{i=1}^{\ell''+r-\ell}|\psi_i|^2=\sum_{i=1}^{(N_0-n)+(\ell'-\ell'')}|\phi_i|^2
+\sum_{i=1}^{N_0-n}|\phi^0_i|^2,
\end{equation}
which implies that there is a unitary matrix $U$ of the appropriate size such that either
\begin{equation}\label{e:normform1sigc1}
\psi U=(\phi,\phi^0,0,\ldots,0),\quad\text{{\rm if $\ell''+r=\ell\geq 2(N_0-n)+(\ell'-\ell'')$}},
\end{equation}
or
\begin{equation}\label{e:normform1sigc2}
(\psi,0\ldots, 0) U =(\phi,\phi^0),\quad\text{{\rm if $\ell''+r=\ell\ < 2(N_0-n)+(\ell'-\ell'')$}},
\end{equation}
where as above zero components are added so that the vectors on each side have the same number of components. The proof that $\deg H'=\deg H^0$ is now completely analogous to the case $\sigma'=1$ above. The details are left to the reader.
\end{proof}

We are now in a position to prove Theorem \ref{thm:main}.

\begin{proof}[Proof of Theorem $\ref{thm:main}$] The conclusion that $H$ is rational follows directly from Lemma \ref{lem:main}. The conclusion that $\deg H=m$ will be a consequence of Proposition \ref{prop:main}. To see this, we first note that $\Omega$ contains the $\ell$-dimensional projective plane $X$ defined by $Z_{\ell+1}=\ldots=Z_{n+1}=0$ as a consequence of \eqref{e:P=0} and, hence, condition (i) in Proposition \ref{prop:main} is satisfied. Condition (ii) is the conclusion of Lemma \ref{lem:M}. Finally, as the rational mapping $H^0$ in condition (iii) we may take the mapping
\begin{equation}\label{e:H^0}
Z=[Z_0:\ldots:Z_{n+1}]\mapsto [Z_0^m:\ldots:Z_\ell^m:P_{\ell+1}(Z_0,\ldots,Z_{n+1}):\ldots:
P_{N_0+1}(Z_0,\ldots,Z_{n+1})].
\end{equation}
Since $\deg H^0=m$, the desired conclusion $\deg H=m$ follows from Proposition \ref{prop:main}. This completes the proof of Theorem \ref{thm:main}.
\end{proof}

%\bibliographystyle{alpha}
%\bibliography{mybib}

\begin{thebibliography}{DKR03}

\bibitem[Ale74]{Alexander74}
H.~Alexander.
\newblock Holomorphic mappings from the ball and polydisc.
\newblock {\em Math. Ann.}, 209:249--256, 1974.

\bibitem[BEH08]{BEH08}
M.~Salah Baouendi, Peter Ebenfelt, and Xiaojun Huang.
\newblock Super-rigidity for {CR} embeddings of real hypersurfaces into
  hyperquadrics.
\newblock {\em Adv. Math.}, 219(5):1427--1445, 2008.

\bibitem[BEH11]{BEH09}
M.~Salah Baouendi, Peter Ebenfelt, and Xiaojun Huang.
\newblock Holomorphic mappings between hyperquadrics with small signature
  difference.
\newblock {\em Amer. J. Math.}, 133(6):1633--1661, 2011.

\bibitem[BER07]{BER07}
M.~Salah Baouendi, Peter Ebenfelt, and Linda~Preiss Rothschild.
\newblock Transversality of holomorphic mappings between real hypersurfaces in
  different dimensions.
\newblock {\em Comm. Anal. Geom.}, 15:589--611, 2007.

\bibitem[BH05]{BH05}
M.~Salah Baouendi and XiaoJun Huang.
\newblock Super-rigidity for holomorphic mappings between hyperquadrics with
  positive signature.
\newblock {\em J. Differential Geom.}, 69:379--398, 2005.

\bibitem[D'A88]{DAngelo88}
John~P. D'Angelo.
\newblock Proper holomorphic maps between balls of different dimensions.
\newblock {\em Michigan Math. J.}, 35(1):83--90, 1988.

\bibitem[D'A91]{Dangelo91}
J.~P. D'Angelo.
\newblock Polynomial proper holomorphic mappings between balls, ii.
\newblock {\em Michigan Math. J.}, 38, 1991.

\bibitem[DKR03]{DAngeloKR03}
John~P. D'Angelo, {\v{S}}imon Kos, and Emily Riehl.
\newblock A sharp bound for the degree of proper monomial mappings between
  balls.
\newblock {\em J. Geom. Anal.}, 13(4):581--593, 2003.

\bibitem[DL09]{JPDLebl09}
John~P. D'Angelo and Ji{\v{r}}{\'{\i}} Lebl.
\newblock Complexity results for {CR} mappings between spheres.
\newblock {\em Internat. J. Math.}, 20(2):149--166, 2009.

\bibitem[DL11]{JPDLebl11}
John~P. D'Angelo and Ji{\v{r}}{\'{\i}} Lebl.
\newblock Hermitian symmetric polynomials and {CR} complexity.
\newblock {\em J. Geom. Anal.}, 21(3):599--619, 2011.

\bibitem[E13]{E13}
Peter Ebenfelt.
\newblock Partial rigidity of degenerate CR embeddings into spheres.
\newblock {\em Adv. Math.}, 239:72–-96, 2013.

\bibitem[EHZ04]{EHZ04}
Peter Ebenfelt, Xiaojun Huang, and Dmitry Zaitsev.
\newblock Rigidity of {CR}-immersions into spheres.
\newblock {\em Comm. Anal. Geom.}, 12(3):631--670, 2004.

\bibitem[EHZ05]{EHZ05}
Peter Ebenfelt, Xiaojun Huang, and Dmitry Zaitsev.
\newblock The equivalence problem and rigidity for hypersurfaces embedded into
  hyperquadrics.
\newblock {\em Amer. J. Math.}, 127(1):169--191, 2005.

\bibitem[ES]{ESh12}
Peter Ebenfelt and Ravi Shroff.
\newblock Partial rigidity of {CR} embeddings of real hypersurfaces into
  hyperquadrics with small signature difference.
\newblock {\em Comm. Anal. Geom.}, to appear;
  http://front.math.ucdavis.edu/1011.1034.

\bibitem[Far82]{Faran82}
James~J. Faran.
\newblock Maps from the two-ball to the three-ball.
\newblock {\em Invent. Math.}, 68(3):441--475, 1982.

\bibitem[Far86]{Faran86}
James~J. Faran.
\newblock The linearity of proper holomorphic maps between balls in the low
  codimension case.
\newblock {\em J. Differential Geom.}, 24(1):15--17, 1986.

\bibitem[Ham05]{Hamada05}
Hidetaka Hamada.
\newblock Rational proper holomorphic maps from {$\bold B^n$} into {$\bold
  B^{2n}$}.
\newblock {\em Math. Ann.}, 331(3):693--711, 2005.

\bibitem[HJ01]{HuangJi01}
Xiaojun Huang and Shanyu Ji.
\newblock Mapping $\mathbb{B}^n$ into $\mathbb{B}^{2n-1}$.
\newblock {\em Inventiones Mathematicae}, 145:219--250, 2001.
\newblock 10.1007/s002220100140.

\bibitem[HJX06]{HuangJiXu06}
Xiaojun Huang, Shanyu Ji, and Dekang Xu.
\newblock A new gap phenomenon for proper holomorphic mappings from {$B^n$}
  into {$B^N$}.
\newblock {\em Math. Res. Lett.}, 13(4):515--529, 2006.

\bibitem[HJY09]{HuangJiYin09}
Xiaojun Huang, Shanyu Ji, and Wanke Yin.
\newblock Recent progress on two problems in several complex variables.
\newblock {\em Proceedings of the ICCM 2007, International Press}, Vol
  I:563--575, 2009.

\bibitem[HJY12]{HuangJiYin12}
Xiaojun Huang, Shanyu Ji, and Wanke Yin.
\newblock On the third gap for proper holomorphic maps between balls.
\newblock {\em Preprint; http://front.math.ucdavis.edu/1201.6440}, 2012.

\bibitem[H{\"o}r90]{Hormander90}
Lars H{\"o}rmander.
\newblock {\em An introduction to complex analysis in several variables},
  volume~7 of {\em North-Holland Mathematical Library}.
\newblock North-Holland Publishing Co., Amsterdam, third edition, 1990.

\bibitem[Hua99]{Huang99}
Xiaojun Huang.
\newblock On a linearity problem for proper holomorphic maps between balls in
  complex spaces of different dimensions.
\newblock {\em J. Differential Geom.}, 51:13--33, 1999.

\bibitem[Hua]{Huang}
Xiaojun Huang.
\newblock Private communication.
\newblock To Ng in October, 2011, and to Ebenfelt in October, 2013.

\bibitem[LP11]{LeblPeters11}
Jiri Lebl and Han Peters.
\newblock Polynomials constant on a hyperplane and CR maps of spheres.
\newblock {\em Illinois J. Math.}, to appear, 2011.

\bibitem[KZ13]{KimZaitsev12}
Sung-Yeon Kim and Dmitri Zaitsev
\newblock Rigidity of CR maps between Shilov boundaries of bounded symmetric domains
\newblock {\em Invent. Math.}, 193(2):409--437, 2013.


\bibitem[KZ]{KimZaitsev13}
Sung-Yeon Kim and Dmitri Zaitsev
\newblock Proper holomorphic mappings between bounded symmetric domains
\newblock preprint, October 2013.

\bibitem[Ng]{Ng13}
Sui-Chung Ng.
\newblock Proper holomorphic mappings on flag domains of {SU}$(p,q)$-type on
  projective spaces.
\newblock {\em Michigan Math. J.}, to appear;
  http://hkumath.hku.hk/~imr/IMRPreprintSeries/2011/IMR2011-12.pdf.

\bibitem[Poi07]{Poincare07}
H. Poincar{\'e}.
\newblock Les fonctions analytiques de deux variables et la repr{\'e}sentation
  conforme.
\newblock {\em Rend. Circ. Mat. Palermo}, 23(2):185--220, 1907.

\bibitem[Web79]{Webster79}
S.~M. Webster.
\newblock The rigidity of {C}-{R} hypersurfaces in a sphere.
\newblock {\em Indiana Univ. Math. J.}, 28(3):405--416, 1979.

\end{thebibliography}

\def\cprime{$'$}

\end{document}